\documentclass[12pt]{amsart}
\usepackage[dvipdfmx]{xcolor,graphicx}
\usepackage{amsmath,amssymb,amsthm,mathrsfs,ascmac,comment,lineno}
\usepackage[top=30truemm,bottom=30truemm,left=30truemm,right=30truemm]{geometry}
\usepackage[all]{xy}
\usepackage[capitalize]{cleveref}

\newcommand{\ctext}[1]{\raise0.2ex\hbox{\textcircled{\scriptsize{#1}}}}
\newcommand{\hi}{\mathchar`-}

\newcommand{\W}[1]{\widetilde{#1}}
\newcommand{\Q}{\mathbb{Q}}
\newcommand{\R}{\mathbb{R}}
\newcommand{\B}{\mathbb{B}}
\newcommand{\Z}{\mathbb{Z}}
\newcommand{\e}{\epsilon}
\newcommand{\de}{\delta}
\newcommand{\s}{\sigma}
\newcommand{\z}{\zeta}
\newcommand{\G}{\textup{Gal}}
\newcommand{\K}{RE_n^+}
\newcommand{\vol}{\textup{vol}}

\theoremstyle{plain}
\newtheorem{defi}{Definition}[section]
\newtheorem{theo}[defi]{Theorem}
\newtheorem{prop}[defi]{Proposition}
\newtheorem{lem}[defi]{Lemma}

\newtheorem{conj}[defi]{Conjecture}

\theoremstyle{definition}

\newtheorem{rem}[defi]{Remark}

\title{Generalized Pell's equations
and Weber's class number problem}

\makeatletter
\@namedef{subjclassname@2020}{
\textup{2020} Mathematics Subject Classification}
\makeatother 

\author{Hyuga Yoshizaki}
\email{yoshizaki.hyuga@gmail.com}
\address{Department of Mathematics, Graduate school of science and Technology, Tokyo University of Science, 2641, Yamazaki, Noda-shi, 278-8510, Chiba, Japan}

\subjclass[2020]{Primary 
11J70, 
11D57, 
11R29; 
Secondary
11R18, 
11R27. 
}

\keywords{Pell's equation, Continued fraction, Weber's class number problem.}

\date{}

\begin{document}


\begin{abstract}
We study a generalization of Pell's equation, whose coefficients are certain algebraic integers.
Let $X_0=0$ and $X_n=\sqrt{2+X_{n-1}}$ for each $n\in \Z_{\ge 1}$.
We study the $\Z[X_{n-1}]$-solutions of the equation $x^2-X_n^2y^2=1$.
By imitating the solution to the classical Pell's equation, we introduce new continued fraction expansions for $X_n$ over $\Z[X_{n-1}]$ and
obtain an explicit solution of the generalized Pell's equation.
In addition, we show that our explicit solution generates all the solutions if and only if the answer to Weber's class number problem is affirmative.
We also obtain a congruence relation for the ratios of the class numbers of the $\Z_2\hi$extension over the rationals and show the convergence of the class numbers in $\Z_2$.
\end{abstract}
\maketitle
\tableofcontents
\section{Introduction}
\label{intro}
{
For a non-square positive integer $m$, {it is well-known that the solutions in integers of Pell's equation
\[
x^2-my^2=1
\]
are given by the regular continued fraction expansion of $\sqrt{m}$ (cf.~\cref{ClassicalMethod}).}
The aim of this paper is to study the $\Z[X_{n-1}]$-solutions of a generalization of Pell's equation:
\begin{equation} \label{gpell}
  x^2-X_n^2y^2=1,
\end{equation}
where $X_n=2\cos(\pi/2^{n+1})$ satisfies $X_0=0$ and $X_{n+1}=\sqrt{2+X_n}$.
For $n=1$, this equation is a classical Pell's equation $x^2-2y^2=1$.}

Now we explain our main results.
First, we give a new continued fraction expansion of $X_n$ (\cref{convX}) as follows;
\[
 {X_n=[1,\overline{2(1+X_{n-1})^{-1},2}]}
\]
We obtain an explicit solution of \cref{gpell} as
\[
(x,y)=(1+2(1+X_{n-1})^{-1},2(1+X_{n-1})^{-1})
\]
by imitating the classical method (cf. Section 2).
We conjecture that our explicit solution ``generates'' all the solutions of \cref{gpell} (\cref{conj}).

{Secondly, we investigate the relation between our conjecture and Weber's class number problem, which asks the class number of $\B_n:=\Q(X_n)$.
The class numbers have been determined to be $1$ for the cases $0\leq n \leq 6$ (see \cite[Theorem 2.1]{Miller} for the case $n=6$) and
there are infinitely many prime numbers that do not divide the class numbers for all $n$ (cf.~\cite{Horie07}, \cite{Fukuda-Komatsu3} and \cite{Morisawa-Okazaki}).
Therefore, it is conjectured that the class number of $\B_n$ is $1$ for all $n$ (Weber's conjecture).
We show that our conjecture is equivalent to  Weber's conjecture  (\cref{C-W}).}
In addition, we show a certain minimality property of the explicit solutions (\cref{min}). 

{Thirdly, we obtain a congruence relation for the class numbers
\[
\frac{h_n}{h_{n-1}}\equiv 1 \pmod{2^n}
\]
for all $n\geq1$ by considering the Galois action on the group generated by our explicit solution (\cref{5-2}).}
By this result, we have that the sequence of the class numbers $\{h_n\}_{n\ge0}$ converges in $\Z_2$.
{This is a rediscovery of Kisilevsky's result \cite[Corollary 2]{Kisilevsky} in a specific case from a different approach (see \cref{Overlap} for details).}

{Finally, we state a conjecture (\cref{L-minConj}) that concerns the ``sizes'' of our explicit solution, and give observations on the conjecture.
By assuming the conjecture, we present a contribution to Weber's conjecture and \cref{conj}.}


\section{Classical method}\label{ClassicalMethod}
In this section, we briefly recall the classical method for Pell's equation (see \cite[Ch.7, \S 7.8]{NMZ} for detail).
For a non-square positive integer $m$, we consider Pell's equation
\begin{equation}\label{classicalPell}
  x^2-my^2=1.
\end{equation}
By mapping $(x,y)$ to $x+\sqrt{m}y$, the solutions of Pell's equation are embedded in $\Z[\sqrt{m}]$, and we set $P_m$ its image.
Since $P_m$ forms a subgroup of the multiplicative group $\Z[\sqrt{m}]^*$ and has a torsion element $-1$, $P_m$ is isomorphic to $\Z/2\Z\oplus\Z$ by Dirichlet's unit theorem.
A {\it fundamental solution} of Pell's equation is defined as a corresponding solution to a generator of $P_m/(\Z/2\Z)\cong\Z$.
It is classically known that a fundamental solution is given by the regular continued fraction of $\sqrt{m}$.

Let
\[
[a_0,a_1,a_2,\dots]=a_0+\cfrac{1}{a_1+\cfrac{1}{a_2+\ddots}}
\]
be a continued fraction $ (a_i\in \Z)$.
Let $p_{-1}=1,p_0=a_0$ and $q_{-1}=0,q_0=1$. For a positive integer $k$, we define $p_k$ and $q_k$ as followings;
\[
p_k=a_kp_{k-1}+p_{k-2},
\]
\[
q_k=a_kq_{k-1}+q_{k-2}.
\]
Then, it holds $p_k/q_k=[a_0,\dots,a_k]$, and the rational number $p_k/q_k$ called the $k$-th convergent of the continued fraction.
It is well-known that the regular continued fraction expansion of $\sqrt{m}$ is of the form
\[
\sqrt{m}=[a_0,\overline{a_1,...,a_l}]:=[a_0,a_1,\dots,a_l,a_1,\dots,a_l,\dots]
\]
and $l$ is called the period of $\sqrt{m}$ if we take the minimal $l$.
Then we obtain a fundamental solution of Pell's equation
\[
  (p,q) = \begin{cases}
    (p_{l-1},q_{l-1}) & (l\text{: even}) \\
    (p_{2l-1},q_{2l-1}) & (l\text{: odd}).
  \end{cases}
\]

{In \cref{Observations}, we observe a characterization of a fundamental solution of \cref{gpell}.
For comparison, we explain why the regular continued fraction expansion of $\sqrt{m}$ gives a fundamental solution.
A solution $(a,b)$ is a fundamental solution if and only if
\begin{equation}\label{ClassicalMinimal}
  |\log |a+\sqrt{m}b||=\min\{|\log|x+\sqrt{m}y|| \mid  x,y\in \Z,\ x^2-my^2=1 \},
\end{equation}
or equivalently,
\[
  |a|=\min\{|x| \in \Z\mid x\neq 1,\ x^2-my^2=1\}.
\]
On the other hand, the regular continued fraction of $\sqrt{m}$ gives a best approximation to $\sqrt{m}$ in the following sense.
\begin{defi}[{best approximation, cf.~\cite[p.9]{Lang}}]
  Let $\alpha$ be an irrational number.
  A best approximation to $\alpha$ is a rational number $p/q$ ($q>0$) such that
  for any rational number $p'/q'$ with $1\leq q' \leq q$, we have
  \[
    |q\alpha-p|<|q'\alpha-p'|.
  \]
\end{defi}
\begin{theo}[{cf.~\cite[Theorem 6]{Lang}}]\label{BA}
  All of best approximations to $\alpha$ are convergents of the regular continued fraction expansion of $\alpha$.
\end{theo}
Let $(x,y)\neq (\pm1,0)$ be a solution of the \cref{classicalPell}.
Then $x/y$ satisfies
\begin{equation}\label{Pell-ineq}
  |\sqrt{m}-\frac{x}{y}|<\frac{1}{2y^2}.  
\end{equation}
If a rational number satisfies the inequality (\ref{Pell-ineq}), then the rational number is a best approximation to $\sqrt{m}$ (cf.~\cite[Corollary 2]{Lang}).
Thus we see that $x/y$ is a convergent of $\sqrt{m}$, and there exists an integer $n$ such that $x/y=p_n/q_n$.
By the theory of continued fraction, if the period $l$ of the regular continued fraction expansion of $\sqrt{m}$ is even (resp. odd),
then $l-1$ (resp. $2l-1$) is the index of the convergent which has the smallest numerator in the set of convergents that can be solutions to \cref{classicalPell},
that is, $(p_{l-1}, q_{l-1})$ (resp. $(p_{2l-1}, q_{2l-1})$) is a fundamental solution.
}

\section{Generalized Pell's equation}
We study the generalized Pell's equation
\[
x^2-X_n^2y^2=1
\]
with the $\Z[X_{n-1}]$-solutions by imitating the classical method.
We obtained a continued fraction expansion of $X_n$ over $\Z[X_{n-1}]$ by a new algorithm.
First, we prepare the algebraic property of $X_n$.


\subsection{Algebraic aspects of $X_n$}
For non-negative integer $n$, set $\B_n=\Q(X_n)$.
Since $X_n=\z_{2^{n+2}}+\z_{2^{n+2}}^{-1}$ we see that $\B_n$ is the maximal real subfield of $\Q(\z_{2^{n+2}})$ where $\z_{2^{n+2}}:=\exp(2\pi\sqrt{-1}/2^{n+2})$.
By the theory of cyclotomic field (see \cite[Ch.2]{Washington} in detail), we have that $\B_n$ is an algebraic number field of degree $2^n$, and Galois extension over $\Q$ with Galois group $\Z/2^n \Z$,
and the ring of integers of $\B_n$ is $\Z[X_n]$.
We see that $\B_n$ is a relative quadratic extension over $\B_{n-1}$.


\subsection{New continued fraction}
We define a new continued fraction expansion algorithm over $\Z[X_{n-1}]$.
For $n\ge 1$, we set $\beta_0=1$ and $\beta_k=2\cos(k\pi/2^n)$ for each $1\leq k \leq 2^{n-1}-1$. 
Then,
\begin{equation}\label{basis}
{\mathcal{B}_{n-1}=\{\beta_k\mid k=0,1,\cdots,2^{n-1}-1 \}}
\end{equation}
 is an integral basis of $\mathbb{Z}[X_{n-1}]$. By embedding
\[
\phi_n:\B_{n-1}\longrightarrow \mathbb{R}^{2^{n-1}};a\mapsto (\tau(a))_{\tau \in \G(\B_{n-1}/\Q)}, 
\]
the basis $\mathcal{B}_{n-1}$ is orthogonal in $\mathbb{R}^{2^{n-1}}$ (cf.~\cite[Lemma 6.3]{MO}), and $\mathbb{Z}[X_{n-1}]$ forms a complete lattice in $\mathbb{R}^{2^{n-1}}$.
Recall $X_n=\sqrt{2+X_{n-1}}$.
We define
\[
\phi_n(X_n)=(\sqrt{2+\tau(X_{n-1})})_{\tau \in \G(\B_{n-1}/\Q)}
\]
and {extend $\phi_n$ to
\[
\phi_n :\B_n\rightarrow \R^{2^{n-1}};a+X_nb\mapsto \phi_n(a)+\phi_n(X_n)\phi_n(b)
\]
for each $a,b \in \B_{n-1}$ where the sum and the multiplication are component-wise.
For each $x\in \R$, let ${\rm round}(x)$ denote the integer in $(x-1/2, x+1/2]$
We note that for each $\alpha \in \B_n$, we can write $\alpha=\sum_{k=0}^{2^{n-1}-1}r_k\beta_k$ uniquely for some $r_k \in \R$.
}

\begin{defi} \label{alg}
For $\alpha=\sum_{k=0}^{2^{n-1}-1}r_k\beta_k \in \B_n$, we define $\lfloor \alpha \rfloor=\sum_{k=0}^{2^{n-1}-1}{\rm round}(r_k)\beta_k \in \Z[X_{n-1}]$ and the sequence $(a_k)_{k\ge0}$ as
\begin{eqnarray*}
&&\alpha_0 =\alpha , a_0=\lfloor \alpha_0 \rfloor, \\
&&\alpha_m =(\alpha_{m-1}-a_{m-1})^{-1}, a_m=\lfloor \alpha_m \rfloor \ \ (m\ge 1).
\end{eqnarray*}
\end{defi}
{\begin{rem}\label{remark}
  By the orthogonality of $\phi_n(\mathcal{B}_{n-1})$, $\phi_n(\lfloor \alpha \rfloor)$ is one of the closest point to $\phi_n(\alpha)$ in $\phi_n(\mathbb{Z}[X_{n-1}])$ for Euclidean distance of $\R^{2^{n-1}}$.
\end{rem}
}

\begin{prop} \label{contX}
Let $\alpha =X_n\in \B_n$. Then we have
\begin{eqnarray*}
a_0&=&1, \\
a_{2k-1}&=&2(1+X_{n-1})^{-1}, \\
a_{2k}&=&2
\end{eqnarray*}
for positive 	integers $k$.
\end{prop}

\begin{proof}
By \cref{remark}, 
it suffices to show that $\phi_n(0)$ is
\begin{itemize}
\item[(a)] a unique closest point to $\phi_n(\sqrt{2+X_{n-1}}-1)$ and
\item[(b)] a unique closest point to $\phi_n((\sqrt{2+X_{n-1}}-1)^{-1}-2(1+X_{n-1})^{-1})=\phi_n((1+\sqrt{2+X_{n-1}})^{-1})$
\end{itemize}
in $\phi_n(\mathbb{Z}[X_{n-1}])$.
For (a), since $\phi_n(\mathcal{B}_{n-1})$ is orthogonal in $\R^{2^{n-1}}$ and the length of $\phi_n(\beta_k)\ (k=1,\cdots ,2^{n-1}-1)$ are $\sqrt{2^n}$ (see \cite[Lemma 6.3]{MO}),
{it is enough to show that}
\begin{itemize}
\item[(a-1)] $||\sqrt{2+X_{n-1}}-1-0||<\sqrt{2^n}/2$ and
\item[(a-2)] $||\sqrt{2+X_{n-1}}-1-0||<||\sqrt{2+X_{n-1}}-1-(\pm1)||$.
\end{itemize}
(a-1). The left-hand side of the inequality is {$||\sqrt{2+X_{n-1}}-1||=\sqrt{2^n \mathfrak{A_n}/\pi}$, where
\[
\mathfrak{A_n}:=\frac{\pi}{2^n}\sum_{k=1}^{2^{n-1}}\left(2\cos\left(\frac{2k-1}{2^{n+1}}\pi \right)-1\right)^2
<\int_{-\frac{\pi}{2^{n+1}}}^{\frac{\pi}{2}+\frac{\pi}{2^{n+1}}}(2\cos x-1 )^2 dx=:I_n
\]
Now $(I_n)_{n\geq6}$ is decreasing with $I_6=0.762...<\pi/4$ and we can check numerically that $\mathfrak{A_n}<\pi/4$ for for the cases $1\le n\le 5$.}

(a-2). We show that
\begin{description}
\item[(a-2-i)] $||\sqrt{2+X_{n-1}}-1||<||\sqrt{2+X_{n-1}}-1-(+1)||$ and

\item[(a-2-ii)]  $||\sqrt{2+X_{n-1}}-1||<||\sqrt{2+X_{n-1}}-1-(-1)||$.
\end{description}
(a-2-i). Transform the inequality as following;
\begin{eqnarray*}
&&\sum_{k=1}^{2^{n-1}}\left(2\cos\left(\frac{2k-1}{2^{n+1}}\pi\right)-1\right)^2< \sum_{k=1}^{2^{n-1}}\left(2\cos\left(\frac{2k-1}{2^{n+1}}\pi\right)-2\right)^2 \\
&\Leftrightarrow&\frac{\pi}{2^{n-1}}\sum_{k=1}^{2^{n-1}}\cos\left(\frac{2k-1}{2^{n+1}}\pi\right)<\frac{3}{4}\pi.
\end{eqnarray*}
Since the proof of the inequality is almost the same as in case (a-1), using a comparison series-integral with $\cos x$, we omit it.

(a-2-ii). Similarly, we see that it suffices to show that
\[
1<\frac{8}{2^n}\sum_{k=1}^{2^{n-1}}\cos\left(\frac{2k-1}{2^{n+1}}\pi\right)
\]
for $n\ge1$.
In fact, we prove a more general case
\[
1<S_N:=\frac{4}{N}\sum_{k=1}^{N}\cos\left(\frac{2k-1}{4N}\pi\right)\ \ (N\ge1).
\]
For $N=1$, we have $S_1=4\cos(\pi/4)=2\sqrt{2}>1$. For $N\ge2$, a comparison series-integral gives that
\[
S_N\ge I_N:=\frac{8}{\pi}\int_{\frac{\pi}{4N}}^{\frac{2N-1}{4N}\pi}\cos x dx.
\]
Since $I_N={8/\pi}(\cos(\pi/(4N))-\sin(\pi/(4N)))$ and the function $x\mapsto \cos x-\sin x$ decreases in $[0,\pi/4]$, we have that $S_N\ge I_N >I_2>1$.

Similarly to the proof of  (a), we separate the proof of (b) into (b-1) and (b-2).

(b-1). We show that $\left|\left|(1+\sqrt{2+X_{n-1}})^{-1}\right|\right|<\sqrt{2^n}/2$,
which means that
\[
\frac{\pi}{2^n} \sum_{k=1}^{2^{n-1}} \left(\frac{1}{2\cos\left(\frac{2k-1}{2^{n+1}}\pi\right)+1}\right)^2<\frac{\pi}{4}.
\]
However, in the proof of  (b-2-ii), we show that $\pi/2^n \sum_{k=1}^{2^{n-1}} \left(2\cos\left(\tfrac{2k-1}{2^{n+1}}\pi\right)+1\right)^{-1}<\pi/4$ and this induces the statement because $2\cos(\frac{2k-1}{2^{n+1}}\pi)+1>1$ for all $1\le k\le 2^{n-1}$.

(b-2). Similarly to the proof of (a-2), we separate the proof into two cases.

(b-2-i). We show that $\left|\left|\left(1+\sqrt{2+X_{n-1}}\right)^{-1}\right|\right|<\left|\left|\left(1+\sqrt{2+X_{n-1}}\right)^{-1}-(-1)\right|\right|$. This is easy because
\begin{eqnarray*}
&&\sum_{k=1}^{2^{n-1}} \left(\left(\frac{1}{2\cos\left(\frac{2k-1}{2^{n+1}}\pi\right)+1}+1\right)^2-\left(\frac{1}{2\cos\left(\frac{2k-1}{2^{n+1}}\pi\right)+1}\right)^2\right) \\
&=&\sum_{k=1}^{2^{n-1}}\left(1+\frac{2}{2\cos\left(\frac{2k-1}{2^{n+1}}\pi\right)+1}\right)>0.
\end{eqnarray*}

(b-2-ii). $\left|\left|\left(1+\sqrt{2+X_{n-1}}\right)^{-1}\right|\right|<\left|\left|\left(1+\sqrt{2+X_{n-1}}\right)^{-1}-(+1)\right|\right|$. Similarly to the proof of (a-2-i), it suffices to show that
\[
\frac{\pi}{2^n} \sum_{k=1}^{2^{n-1}}\frac{1}{2\cos\left(\frac{2k-1}{2^{n+1}}\pi\right)+1}<\frac{\pi}{4}.
\]
Since the proof of the inequality is almost the same as in case (a-1), using a comparison series-integral with $1/(2\cos x+1)$, we omit it.
\end{proof}

\cref{contX} only provides a formal expansion. We see that it does converge.

\begin{theo} \label{convX}
For $n\ge 1$ and each $\tau \in \G(\B_{n-1}/\Q)$, we have
\[
\sqrt{2+\tau(X_{n-1})}=[1,\overline{2(1+\tau(X_{n-1}))^{-1},2}].
\]
Here, $[a_0,a_1,...]$ denotes $a_0+\cfrac{1}{a_1+\cdots}$ and
$[a_0,\dots,a_r,\overline{a_{r+1},\dots,a_{s}}]$ denotes the periodicity of the part $a_{r+1},\dots,a_{s}$, namely
\[
[a_0,\dots,a_r,\overline{a_{r+1},\dots,a_{s}}]=[a_0,\dots,a_r,a_{r+1},\dots,a_{s},a_{r+1},\dots,a_{s},a_{r+1},\dots,a_{s},\dots].
\]
\end{theo}

\begin{rem}
  {\cref{convX} states that the above continued fraction converges in Euclidean space $\R^{2^{n-1}}\stackrel{\phi_n}{\hookleftarrow}\B_n$.
  Namely we get a continued fraction expansion of $\sqrt{2+X_{n-1}}$ over $\mathbb{Z}[X_{n-1}]$ for each metric induced by $\tau \in \G(\B_{n-1}/\Q)$.
  We could not make sure whether this algorithm gives a continued fraction expansion of any element of $\B_n$,
  and whether this algorithm terminates for any element of $\B_{n-1}$.}
\end{rem}

\begin{proof}
If the continued fraction $[1,\overline{2(1+\tau(X_{n-1}))^{-1},2}]$ converges, then we see that the numerical value of it is $\sqrt{2+\tau (X_{n-1})}$ by an easy calculation.
We show the convergence of $[1,\overline{2(1+\tau(X_{n-1}))^{-1},2}]$ for each $\tau \in \G(\B_{n-1}/\Q)$.
We check the conditions in \cite[Theorem 4.3]{BEJ}.
For $a\in \mathbb{C}$, we define
\[
D(a)=\begin{pmatrix} a & 1 \\ 1 & 0 \end{pmatrix}.
\]
For a continued fraction $[a_1, a_2,...,a_k]$, we define
\[
M([a_1, a_2,...,a_k])=D(a_1)D(a_2)\cdots D(a_k).
\]
We should check the followings for all $n\ge1$ and $\tau \in \G(\B_{n-1}/\Q)$;
\begin{itemize}
\item[(a)]$M([1,2(1+\tau(X_{n-1}))^{-1},2,0,-1,0])\neq \pm\begin{pmatrix}1, 0 \\0,1\end{pmatrix}$
\item[(b)] $|M([2(1+\tau(X_{n-1}))^{-1},2])_{2,2}|\le 1$
\item[(b')] $|M([2,2(1+\tau(X_{n-1}))^{-1}])_{2,2}|\le 1$
\item[(c)] ${\rm Tr}(M([1,2(1+\tau(X_{n-1}))^{-1},2,0,-1,0]))^2\ge4$
\end{itemize}
where $M_{2,2}$ denotes the $(2,2)$-element of a matrix $M$.
The first three (a), (b), and (b') are trivial. 
We note that
\[
  {{\rm Tr}(M([1,2(1+\tau(X_{n-1}))^{-1},2,0,-1,0]))^2}=4(2(1+\tau(X_{n-1}))^{-1}+1)^2.
\]
If $\tau(X_{n-1})>-1$, then we have $(2(1+\tau(X_{n-1}))^{-1}+1)^2\ge1$ and (c) holds.
Otherwise, we have that $-2<\tau(X_{n-1})<-1$. So we have $(1+\tau(X_{n-1}))^{-1}<-1$ and an easy calculation shows that (c) holds.
\end{proof}

In the case $n=1$, the above theorem states that $\sqrt{2}=[1,\overline{2,2}]$ and this is a classical continued fraction expansion of $\sqrt{2}$.


\subsection{$\Z[X_{n-1}]\hi$solutions}\label{exsol}
By imitating the classical method, we formulate a conjecture for the $\Z[X_{n-1}]\hi$solutions of the generalized Pell's equation.
Since the period of $[1,\overline{2(1+X_{n-1})^{-1},2}]$ is $2$, we look at the first convergent
\[
\frac{p_1}{q_1}=\frac{1+2(1+X_{n-1})^{-1}}{2(1+X_{n-1})^{-1}}.
\]
It is easy to check that
\[
p_1^2-X_n^2q_1^2=1
\]
for all $n\ge1$.
We set
\[
\e_n=p_1+X_nq_1.
\]
We conjecture that the element $\e_n$ generates the $\Z[X_{n-1}]\hi$solutions as a Galois module.
\begin{conj}\label{conj}
  The $\Z[X_{n-1}]$-solutions of the generalized Pell's equation $x^2-X_n^2y^2=1$ is a $\G(\B_n/\Q)$-module generated by $-1$ and $\e_n$, namely, 
\[
\{a+X_nb\mid a,b\in \Z[X_{n-1}],a^2-X_n^2b^2=1\}=\langle -1, \e_n \rangle_{\Z[\G(\B_n/\Q)]}.
\]
\end{conj}


\section{Weber's class number problem}\label{Weber}
The aim of this section is to prove the following equivalence:
\begin{theo}\label{C-W}
\cref{conj} is true for all $n\ge0$ if and only if Weber's conjecture is true for all $n\leq 0$.
\end{theo}


\subsection{Some known results}

We prepare some known results.
Let $E_n$ be the group of units of $\B_n$ and

\[
C_n:=\left\langle -1, \z_{2^{n+2}}^{\frac{1-a}{2}}\frac{1-\z_{2^{n+2}}^a}{1-\z_{2^{n+2}}} \mid a:\text{odd integers such that}\ 1<a<2^{n+1} \right\rangle_{\Z}
\]
be its subgroup of cyclotomic units. Then $(E_n:C_n)=h_n$, by  \cite[Lemma 8.1 and Theorem 8.2]{Washington}.
Noticing that $3$ is a generator of {$(\Z/2^{n+2}\Z)^*/\{\pm1\}$} and that
\[
1+X_n=\z_{2^{n+2}}^{\frac{1-3}{2}}\frac{1-\z_{2^{n+2}}^3}{1-\z_{2^{n+2}}},
\]
by \cite[Proposition 8.11]{Washington}, we have
\[
C_n=\langle 1+X_n \rangle_{\Z[\G(\B_n/\Q)]}.
\]

We set $G_{n/n-1}=\G(\B_n/\B_{n-1})$ and define $\sigma_{n/n-1}$ to be the non-trivial element of $G_{n/n-1}$.
We note that $\s_{n/n-1}(X_n)=-X_n$.
We define a relative norm map by 
\[
N_{n/n-1}:\B_n\longrightarrow \B_{n-1};x\mapsto x\s_{n/n-1}(x).
\]
\begin{lem}\label{lem}
The restrictions $N_{n/n-1}|_{E_n}:E_n\longrightarrow E_{n-1}$ and $N_{n/n-1}|_{C_n}:C_n\longrightarrow C_{n-1}$ are well-defined and surjective.
\end{lem}

\begin{proof}

Let $\hat{H}^r(G_{n/n-1},E_n)$ be the $r\hi$th Tate cohomology group. It suffices to show that $\hat{H}^0(G_{n/n-1},E_n)=\{1\}$ for the surjectivity of $N_{n/n-1}|_{E_n}:E_n\longrightarrow E_{n-1}$.
Yokoi \cite[Lemma 3]{Yokoi} showed that
\[
Q(E_n)=\frac{|\hat{H}^0(G_{n/n-1},E_n)|}{|\hat{H}^1(G_{n/n-1},E_n)|}=\frac{1}{2}.
\]
Therefore, it suffices to show that $|\hat{H}^1(G_{n/n-1},E_n)|=2$. Let $H_{n-1}$ be the maximal unramified abelian extension of $\B_{n-1}$. Then we have $\B_n\cap H_{n-1}=\B_{n-1}$ because $\B_n/\B_{n-1}$ ramifies at {the prime ideal lying above $2$}. Furthermore, $\B_n/\B_{n-1}$ ramifies at only one prime, then $\B_n/\B_{n-1}$ satisfied the assumption of \cite[Theorem 1]{Yokoi}. Thus we have $h_{n-1}=|\text{Cl}_n^{G_{n/n-1}}|$.
Since we have $2\nmid h_{n-1}$ by Weber, we get $|\hat{H}^1(G_{n/n-1},E_n)|=2$ by the Corollary of  \cite[Theorem 2]{Yokoi}.
Thus we see that $N_{n/n-1}:E_n\longrightarrow E_{n-1}$ is surjective.

Next we consider $N_{n/n-1}|_{C_n}$.
The presentation $C_n=\langle 1+X_n \rangle_{\Z[\G(\B_n/\Q)]}$ and the easy calculations $N_{n/n-1}(1+X_n)=-1-X_{n-1}$ and $N_{n/n-1}((1+X_n)\s(1+X_n)\cdots \s^{2^{n-1}-1}(1+X_n))=-1$ show that $N_{n/n-1}:C_n\longrightarrow C_{n-1}$ is well-defined and surjective, where $\s$ is a generator of $\G(\B_n/\Q)$.
\end{proof}


Set $\K=\ker (N_{n/n-1}|_{E_n})$ throughout this paper.
\cref{lem} induces
the following exact sequence;
\begin{equation}\label{exact}
0\rightarrow \K/A_n \rightarrow E_n/C_n\rightarrow E_{n-1}/C_{n-1}\rightarrow 0
\end{equation}
where $A_n:=\K\cap C_n$. By the exact sequence (\ref{exact}), Weber's conjecture is equivalent to

\begin{equation}\label{K}
(\K:A_n)=1\ \text{for all}\ n\ge1.
 \end{equation}


\subsection{Proof of \cref{C-W}}
For $\epsilon \in \K$, there exist unique $a,b\in \mathbb{Z}[X_{n-1}]$ such that $\epsilon =a+bX_n$ and we have $N_{n/n-1}(\epsilon)=a^2-b^2X_n^2$.
Thus
we have a bijection;
\[
\begin{array}{ccc}
\K & \stackrel{}{\longleftrightarrow} & \{\text{the solutions of the generalized Pell's equation }x^2-X_n^2y^2=1\} \\
\rotatebox{90}{$\in$} & & \rotatebox{90}{$\in$} \\
\e=a+X_nb & \longleftrightarrow & (a,b)
\end{array}
\]
We recall that \cref{conj} states
\[
\{a+X_nb\mid a,b\in \Z[X_{n-1}],a^2-X_n^2b^2=1\}=\langle -1, \e_n \rangle_{\Z[\G(\B_n/\Q)]}.
\]
{Therefore, \cref{conj} is equivalent to that $\K=\langle -1, \e_n \rangle_{\Z[\G(\B_n/\Q)]}$ for all $n$.
Combining this formulation and (\ref{K}), to prove \cref{C-W},}
it suffices to prove that 
\[
A_n=\langle -1, \e_n \rangle_{\Z[\G(\B_n/\Q)]}.
\]

By easy calculation, we have that
\[
\e_n=\frac{X_n+1}{X_n-1}
\]
for each $n\ge1$. Since $C_n=\langle -1, 1+X_n \rangle_{\Z[\G(\B_n/\Q)]}$, we have $\e_n\in C_n$ and $\e_n\in C_n \cap \K=A_n$. Thus we have $\langle -1, \e_n \rangle_{\Z[\G(\B_n/\Q)]}\subset A_n$.

We put $\W{N}_{n/n-1}|_{C_n}:C_n/\{\pm 1\}\longrightarrow C_{n-1}/\{\pm 1\}$.
Let $\s$ be a generator of $\G(\B_n/\Q)$.
We note that the basis of $C_n/\{\pm 1\}$ and $C_{n-1}/\{\pm 1\}$ are $\{\s(1+X_n), \s^2(1+X_n), \dots, \s^{2^n-1}(1+X_n)\}$ and  $\{\s(1+X_{n-1}), \s^2(1+X_{n-1}), \dots, \s^{2^{n-1}-1}(1+X_{n-1})\}$ respectively.
By considering the representation matrix of  $\W{N}_{n/n-1}|_{C_n}$, we see that the basis of the kernel of $\W{N}_{n/n-1}|_{C_n}$ is
\[
\left\{\s^i\left( \frac{1+X_n}{1-X_n}\right) , \prod_{j=0}^{2^{n-1}-1}\s^j\left( 1+X_n \right) \mid i=1,2,...,2^{n-1}-1 \right\}.
\]
Since $\sigma^i\left(\frac{1+X_n}{1-X_n}\right)\in \left\langle -1,\frac{X_n+1}{X_n-1} \right\rangle_{\mathbb{Z}[G_n]}$, the rest of the proof is showing that
\[
\prod_{j=0}^{2^{n-1}-1}\sigma^{j}(1+X_n)^e\in \left\langle -1, \frac{X_n+1}{X_n-1} \right\rangle_{\mathbb{Z}[G_n]}
\]
if $N_{n/n-1}\left( \prod_{j=0}^{2^{n-1}-1}\sigma^{j}(1+X_n)^e\right)=1$. Such $e$ is even because 
\[
N_{n/n-1}\left(\prod_{j=0}^{2^{n-1}-1}\sigma^{j}(1+X_n)^{e}\right)=\left( \prod_{j=0}^{2^n-1}\sigma^j(1+X_n)\right)^{e}=(-1)^{e}.
\]
Therefore it suffices to show that $ \prod_{j=0}^{2^{n-1}-1}\sigma^{j}(1+X_n)^2\in \left\langle -1,\frac{X_n+1}{X_n-1} \right\rangle_{\mathbb{Z}[G_n]}$. Since
$\prod_{j=0}^{2^n-1}\sigma^j(1+X_n)=-1$, we have
\[
\prod_{j=2^{n-1}}^{2^n-1}\sigma^j(1+X_n)^2
=-\prod_{j=0}^{2^{n-1}-1}\sigma^{j}\left(\frac{1+X_n}{1-X_n}\right)\in \left\langle -1,\frac{X_n+1}{X_n-1} \right\rangle_{\mathbb{Z}_[G_n]}.
\]
Then the assertion follows.


\section{Results on the explicit unit $\e_n$}
In this section, first we show the ``minimality'' of our explicit unit $\e_n$.
Secondly, from the Galois action on relative units and the explicitness of $\e_n$, we obtain a congruence relation formula for the ratios of the class numbers.

\subsection{The minimality of $\e_n$ in $\K$ }
For $n=1$, $\e_1=3+2\sqrt{2}$ comes from the continued fraction of $\sqrt{2}$. By the classical method, we have that $\e_1$ generates all the $\Z$-solutions of Pell's equation $x^2-2y^2=1$. This means that $\e_1$ is ``minimal'', that is,
\[
\e_1^{\frac{l}{m}}\not \in E_{1/0} \text{ for any reduced fraction } \frac{l}{m} \text{ with } {0<|\frac{l}{m}|<1}.
\]
It follows that Weber's conjecture for $n=1$.
We show that $\e_n$ is also ``minimal'' for $n\ge2$.

\begin{theo}\label{min}
$\e_n^{\frac{l}{m}}\not \in \K$ for any reduced fraction $\frac{l}{m}$ with {$0<|\frac{l}{m}|<1$}.
\end{theo}
\begin{proof}
Let $n\ge2$.
It suffices to show the statement in case $\frac{l}{m}=\frac{1}{p}$ for each prime $p$.
We separate the proof into two cases $p=2$ or an odd prime.

Suppose $p=2$.
If $\e_n^{1/2}\in \K\subset \mathbb{B}_n$, then its conjugates are also included in $\mathbb{B}_n$. 
For $\tau \in \G(\B_n/\Q)$, $\tau\left(\sqrt{\frac{X_n+1}{X_n-1}}\right)^2=\frac{\tau(X_n)+1}{\tau(X_n)-1}$. On the other hand, there exists $\tau \in \G(\B_n/\Q)$ such that
$0<\tau(X_n)<1$.
For such $\tau$, we have $\tau\left(\sqrt{\frac{X_n+1}{X_n-1}}\right)^2<0$ and this contradicts the fact that $\B_n$ is a totally real field.
Thus we have $\e_n^{1/2}\not \in \K$.

Now assume that $p\ge3$. By \cite[Proposition 6.6]{MO} for $n\ge2$ and $\pm1 \neq \delta \in \K$ we have 
\begin{equation}\label{MO}
{\rm Tr}_n(\delta^2)\ge 2^n\cdot17
\end{equation}
where ${\rm Tr}_n:\mathbb{B}_n\longrightarrow \mathbb{Q}$ be the trace map of $\B_n$.

Suppose that $\e_n^{1/p} \in \K$. For each $\tau \in \G(\B_n/\Q)$, the conjugate of $\e_n^{1/p}$ is $\left(\frac{\tau(X_n+1)}{\tau(X_n-1)}\right)^{1/p}$. Then we have
\[
{\rm Tr}_n\left(\epsilon^{\frac{2}{p}}\right)=\sum_{k=1}^{2^n}f_p\left(\frac{2k-1}{2^{n+1}}\pi\right),\text{ where }f_p(x):=\left| \frac{2\cos x+1}{2\cos x-1} \right|^{\frac{2}{p}}.
\]
Since $\left|2\cos((2k-1)\pi/2^{n+1})+1\right|<\left|2\cos((2k-1)\pi/2^{n+1})-1\right|$ for $k=2^{n-1}+1,\ldots,2^{n}$,
we have $f_p((2k-1)\pi/2^{n+1})<1$ for such $k$.
Therefore, by using (\ref{MO}) it suffices to show that $\sum_{k=1}^{2^{n-1}}f_p((2k-1)\pi/2^{n+1})<2^{n-1}\cdot17$.

For $k=1,\ldots,2^{n-1}$, we have $\left| (2\cos((2k-1)\pi/2^{n+1})+1)/(2\cos((2k-1)\pi/2^{n+1})-1) \right|>1$. Then we have 
\[
f_p\left(\frac{2k-1}{2^{n+1}}\pi\right)<f_3\left(\frac{2k-1}{2^{n+1}}\pi\right)
\]
for $p>3$. Therefore it suffices to show this in case $p=3$.  Thus our goal is to show that 
\[
\frac{1}{2^{n}}\sum_{k=1}^{2^{n-1}}f_3\left(\frac{2k-1}{2^{n+1}}\pi\right)<\frac{17}{2}
\]
for $n\ge2$.
Let $K$ be the integer satisfied $(2K-1)\pi/2^{n+1}<\pi/3<(2K+1)\pi/2^{n+1}$.
We write
\begin{equation}
\begin{split}
  \frac{1}{2^n}\sum_{k=1}^{2^{n-1}}f_3\left(\frac{2k-1}{2^{n+1}}\pi\right)&=\frac{1}{2^n}\sum_{k=1}^{K-1}f_3\left(\frac{2k-1}{2^{n+1}}\pi\right)+\frac{1}{2^n}f_3\left(\frac{2K-1}{2^{n+1}}\pi\right) \\
&+\frac{1}{2^n}f_3\left(\frac{2K+1}{2^{n+1}}\pi\right)+\frac{1}{2^n}\sum_{k=K+2}^{2^{n-1}}f_3\left(\frac{2k-1}{2^{n+1}}\pi\right).
\end{split}
\end{equation}
A comparison series-integral gives that
\begin{equation}
\begin{split}
&\frac{\pi}{2^n}\sum_{k=1}^{K-1}f_3\left(\frac{2k-1}{2^{n+1}}\pi\right)+\frac{\pi}{2^n}\sum_{k=K+2}^{2^{n-1}}f_3\left(\frac{2k-1}{2^{n+1}}\pi\right) \\
&<\int_{0}^{\pi/3}f_3(x)dx+\int_{\pi/3}^{\pi/2}f_3(x)dx=6.4669....
\end{split}
\end{equation}
{We used a computer for the last integral calculations}.

{
Finally, we claim that
\[
\frac{1}{2^n}f_3\left(\frac{2K-1}{2^{n+1}}\pi\right)+\ \frac{1}{2^n}f_3\left(\frac{2K+1}{2^{n+1}}\pi\right){<3}.
\]
for $n\geq 2$.
Indeed, a function $x\mapsto x\frac{2\cos(\pi/3+x)+1}{2\cos(\pi/3+x)-1}$ is increasing from $-\pi$ to $0$ on $[-\pi/3,\pi/3]$. So we have $f_3(\pi/3+x)\le (\pi/|x|)^{2/3}$ on $[-\pi/3,\pi/3]\setminus \{0\}$.
Set $r=2^{n+1}+3-6K$. So we see that $r\in\{1,5\}$, $\frac{2K-1}{2^{n+1}}\pi=\frac{\pi}{3}-\frac{r}{3\cdot 2^{n+1}}\pi$ and $\frac{2K+1}{2^{n+1}}\pi=\frac{\pi}{3}+\frac{6-r}{3\cdot 2^{n+1}}\pi$. Since $\frac{5}{3\cdot 2^{n+1}}\pi<\frac{\pi}{3}$ for $n\ge2$, we obtain that
\begin{equation}
\begin{split}
&\frac{1}{2^n}f_3\left(\frac{2K-1}{2^{n+1}}\pi\right)+\ \frac{1}{2^n}f_3\left(\frac{2K+1}{2^{n+1}}\pi\right) \\
&\le \frac{1}{2^n}\left(\frac{\pi}{\frac{1}{3\cdot 2^{n+1}}\pi}\right)^{2/3}+\frac{1}{2^n}\left(\frac{\pi}{\frac{5}{3\cdot 2^{n+1}}\pi}\right)^{2/3} \\
&= {2^{\frac{2-n}{3}}(3^{\frac{2}{3}}+\left(\frac{3}{5}\right)^{\frac{2}{3}})<3}
\end{split}
\end{equation}
for $n\geq 2$.}
Thus we have the claim and the assertion holds.
\end{proof}
\begin{rem}
  For $n=2$, we also show that $h_2=1$ by a similar method used above.
  Let $\s$ be a generator of $\G(\B_2/\Q)$.
  Since $h_1=1$, we have $h_2=(RE_2^+:A_2)$. We recall that $A_2=\langle -1,\e_2 \rangle_{\Z[\G(\B_2/\Q)]}$ and note that $(RE^+_2:A_2)<\infty$. We should show that $ \e_2^x\cdot\s\left( \e_2  \right)^y \not \in RE^+_2$ for any $x,y\in [-1/2,1/2]\cap \Q$ except for $x=y=0$. If $\e_2^x\cdot\s\left( \e_2  \right)^y \in \B_2$, then we have
  \[
  {\rm Tr}_2\left( \e_2^{2x}\cdot\s\left( \e_2  \right)^{2y}  \right)=\sum_{i=1}^{4}\s^i\left( \e_2  \right)^{2x}\cdot\s^{i+1}\left( \e_2  \right)^{2y} .
  \]
  Now we define a function $f_2(x,y)={\rm Tr}_2\left(\e_2^{2x}\cdot\s\left( \e_2  \right)^{2y}  \right)$ on $[-1/2,1/2]^2$. Since $\frac{\partial^2f_2}{\partial x^2}(x,y)$ (resp. $\frac{\partial^2f_2}{\partial y^2}(x,y))>0$ for each $y$ (resp. $x$)$\in[-1/2,1/2]^2$ and $f_2\left(\pm1/2,0\right)=f_2\left(0,\pm1/2\right)<f_2\left(\pm1/2,\pm1/2\right)$, the maximum of $f_2(x,y)$ is taken at the points $\left(\pm1/2,\pm1/2\right)$.
  We have $f_2\left(\pm1/2,\pm1/2\right)=28<2^2\cdot17$. This contradicts (\ref{MO}), so we have $RE^+_2=A_2$ and $h_2=1$.
\end{rem}

\subsection{The ratios of the class numbers}\label{gal}
We define the {\it relative class ratio} of $\B_n/\B_{n-1}$ by
\[
k_n=\frac{h_n}{h_{n-1}}
\]
for each $n> 0$.
In this subsection, we obtain a congruence relation formula for $k_n$.

{By (\ref{exact}) in Section 4}, we have 
\[
k_n=\left(\K:A_n\right).
\]
For each prime $l$, let $\left(\K/A_n\right)_l$ denotes the Sylow $l\hi$subgroup of $\K/A_n$, that is the subgroup consisting of elements of $l\hi$power order.
Let $(k_n)_l=\left|\left(\K/A_n\right)_l\right|$. The next theorem is our second main theorem.
\begin{theo} \label{5-2}
For all prime $l$ and all positive integer $n$, we have
\[
(k_n)_l\equiv 1\pmod{2^n}.
\]
\end{theo}

This theorem shows that the sequence $\{h_n\}$ is a Cauchy sequence in $2\hi$adic topology.
Thus the sequence $\{h_n\}$ converges in $\Z_2$.

\begin{rem}\label{Overlap}
{Kisilevsky also obtained the convergence of the class numbers for more general setting in \cite[Corollary 2]{Kisilevsky}.
He showed that for any $\Z_p$-extension over any global field, the sequence of the class numbers of the intermediate fields converges in $\Z_p$.
He used the direct limit of the class groups instead of the unit groups, and the proof is different from ours.
We give an extensive numerical study of the $p$-adic limits for elliptic curves and knots in \cite{Ueki-Yoshizaki}.
}
\end{rem}
We prepare two lemmas.
We note that $\G(\B_n/\Q)$ acts on $\K/A_n$ and also on $\left(\K/A_n\right)_l$.

\begin{lem} \label{5-1}
For $\delta \in \K/A_n$, let $O(\delta)$ be the $\G(\B_n/\Q)\hi$orbit of $\delta$ in $\K/A_n$.
If $|O(\delta)|<2^n$, then $\delta^2=1$ in $\K/A_n$.
\end{lem}
\begin{proof}
$|O(\delta)|<2^n$ means $\s^{2^{n-1}}(\delta)= \delta$ in $\K/A_n$. Therefore, we have $N_{n/n-1}(\delta)= \delta \s^{2^{n-1}}(\delta)= \delta^2$ in $\K/A_n$. On the other hand, 
since $\delta \in \K$, we have $N_{n/n-1}(\delta)= 1$ in $\K/A_n$. Then we have $\delta^2= 1$ in $\K/A_n$.
\end{proof}

\begin{lem} \label{5-3}
Let $\de \in \K/A_n$. If $|O(\de)|=1$, then $\de = 1$ in $\K/A_n$.
\end{lem}
\begin{proof}
Set $\e=(X_n+1)/(X_n-1)$ (abbreviate ``$n$'').
Suppose that there exists $\de \in \K/A_n$ with $\de \neq 1$ in $\K/A_n$ and $|O(\de)|=1$. 
By \cref{5-1}, we have $\de^2\in A_n$. Since $A_n=\left\langle -1,\e \right\rangle_{\Z[\G(\B_n/\Q)]}$, $\de^2$ can be represented as 
\[
\pm \e^{e_0}\s\left(\e\right)^{e_1}\dots \s^{2^{n-1}-1}\left(\e\right)^{e_{2^{n-1}-1}}
\]
by certain integers $e_i$. Therefore, we have 
\[
\de=\pm \sqrt{\left| \e^{e_0}\s\left(\e\right)^{e_1}\cdots \s^{2^{n-1}-1}\left(\e\right)^{e_{2^{n-1}-1}}\right|}.
\]
On the other hand, $|O(\de)|=1$ implies $\s(\de)= \de$ in $\left(\K/A_n\right)_2$. Therefore, we have
\begin{eqnarray*}
&&\sqrt{\left| \e^{e_0}\s\left(\e\right)^{e_1}\cdots \s^{2^{n-1}-1}\left(\e\right)^{e_{2^{n-1}-1}}\right| } \\
&=&\sqrt{\left| \s\left(\e\right)^{e_0}\s^2\left(\e\right)^{e_1}\cdots \s^{2^{n-1}}\left(\e\right)^{e_{2^{n-1}-1}}\right| } \\
&=&\sqrt{\left| \e^{-e_{2^{n-1}-1}}\s\left(\e\right)^{e_0}\cdots \s^{2^{n-1}-1}\left(\e\right)^{e_{2^{n-1}-2}}\right| } \ \text{in $\left(\K/A_n\right)_2$}.
\end{eqnarray*}
Note that $\s^{2^{n-1}}\left(\e\right)=\e^{-1}$. 
Since $\left\{ \e,\s(\e),\dots,\s^{2^{n-1}-1}(\e) \right\}$ are linearly independent over $\Z$ in $\K$,
we have 
\[
-e_{2^{n-1}-1}\equiv e_0\equiv e_1\equiv \dots \equiv e_{2^{n-1}-2}\equiv e_{2^{n-1}-1} \pmod 2.
\]
This implies that $e_i\equiv0 \pmod 2$ for all $i$ or $e_i\equiv1 \pmod 2$ for all $i$. Since $\de\neq 1$, we have $e_i=1$ for all $i$.
Then we have $\sqrt{\left| \e\s(\e)\dots \s^{2^{n-1}-1}(\e)\right|} \in \K$. 

By easy calculation, we have 
\[
(X_n+1)\s(X_n+1)\cdots \s^{2^{n-1}-1}(X_n+1)(X_n-1)\s(X_n-1)\cdots \s^{2^{n-1}-1}(X_n-1)=-1.
\]
It follows that
\[
-\e\s(\e)\cdots \s^{2^{n-1}-1}(\e)=\left( \frac{1}{(X_n-1)\cdots \s^{2^{n-1}-1}(X_n-1)} \right)^2.
\]
Thus we have
\[
\sqrt{\left| \e\s(\e)\cdots \s^{2^{n-1}-1}(\e)\right|} =\left| \frac{1}{(X_n-1)\cdots \s^{2^{n-1}-1}(X_n-1)} \right|.
\]
Since $N_{n/n-1}\left((X_n-1)\cdots \s^{2^{n-1}-1}(X_n-1)\right)=-1$ and $N_{n/n-1}(-1)=1$, we have $N_{n/n-1}(\left| \frac{1}{(X_n-1)\cdots \s^{2^{n-1}-1}(X_n-1)} \right|)=-1$. 
This contradicts
$\sqrt{\left| \e \s\left(\e\right)\cdots \s^{2^{n-1}-1}\left(\e\right)\right|}$
$ \in \ker N_{n/n-1}$. 
\end{proof}

\begin{proof}[Proof of \cref{5-2}]
First, we prove this for an odd prime $l$.
Suppose that there exists an element $\delta \neq 1$ in $\left(\K/A_n\right)_l$ such that $|O(\delta)|<2^n$. By  \cref{5-1}, the order of $\delta$ is $2$. This contradicts $2\nmid \left|\left(\K/A_n\right)_l\right|$. Therefore, all elements except $1$ in $\left(\K/A_n\right)_l$ have $2^n$ distinct conjugates. This implies the statement.

Next, we consider the case $l=2$, independently of Weber's proof.
Suppose that there exists an element $\de \neq 1$ in $\left(\K/A_n\right)_2$ such that $|O(\de)|<2^n$ and we see that $|O(\de)|>1$ by \cref{5-3}. Let $\de$ be an element with the smallest size of $|O(\de)|=2^m$. We note that $\de$ satisfies $\s^{2^m}(\de)= \de$ and $\s^{2^{m-1}}(\de)\neq \de$ in $\left(\K/A_n\right)_2$.
Since $\s^{2^{m-1}}(\de\s^{2^{m-1}}(\de))=\s^{2^{m-1}}(\de)\s^{2^m}(\de)= \s^{2^{m-1}}(\de)\de$ in $\left(\K/A_n\right)_2$, we have $|O(\de\s^{2^{m-1}}(\de))|\le 2^{m-1}$. By the assumption, we have that $\de\s^{2^{m-1}}(\de)= 1$ and $\de=\s^{2^{m-1}}(\de)^{-1}$ in $\left(\K/A_n\right)_2$.
By \cref{5-1}, we have $\s^{2^{m-1}}(\de)^{-1}= \s^{2^{m-1}}(\de)$ in $\left(\K/A_n\right)_2$.
Thus we have $\de=\s^{2^{m-1}}(\de)$ in $\left(\K/A_n\right)_2$ and this is a contradiction.
\end{proof}

\begin{rem}
  By \cref{5-2}, we have $2\nmid h_n$ for all $n\ge1$. This result was first proved by Weber \cite[Theorem C]{HWeber}, but the proof we have now given is independent of the one by Weber.
  In the  proof of \cref{5-2}, we use the fact that $N_{n/n-1}:E_n/C_n\longrightarrow E_{n-1}/C_{n-1}$ is surjective and it comes from $2\nmid h_{n-1}$ (see the proof of \cref{lem}).
  Therefore it may seem like a tautology, but if we admit $h_0=h(\Q)=1$, the proof goes well by induction without using Weber's result. 
  Moreover, our result is a much more refined version of Weber's result.
\end{rem}

\begin{rem}
  Recall that $h_6=1$, then we have $(k_7)_l=(h_7)_l$.
  By \cref{5-2}, we have
  \[
  (h_n)_l\equiv 1\pmod{2^7}
  \]
  for all odd prime $l$ and positive integer $n$.
\end{rem}

\section{Observations on the sizes of $\e_n$}\label{Observations}

In this section, by imitating the classical Pell's equation, we observe some ``sizes'' of the explicit unit $\e_n$
and state the conjecture on the minimality of $\e_n$.
By assuming the conjecture, we give an upper bound for $k_n$ for small $n$.

By embedding $l: \K\to \R^{2^{n-1}}; \e \mapsto (\log |\s^{i}(\e)|)_i$, $l(\K)$ forms a complete lattice in $\R^{2^{n-1}}$.
For a positive integer $p$, let $||x||_p=(\sum_{i=1}^{2^{n-1}}|x_i|^p)^{1/p}$ denote the $L^p$ norm of $x$ in $\R^{2^{n-1}}$.

\begin{defi}[$L^p$-minimal]
  Let $S$ be a subset of $\K$.
  For $\e \in \K\setminus\{\pm1\}$, if
  $l(\e)$ has a minimal $L^p$ norm in $l(S\setminus\{\pm1\})$,
  then $e$ is said to be $L^p$-minimal in $S$.
\end{defi}
In the case of $n=1$, if $\e \in RE_1^+$ corresponds to a fundamental solution, then $\e$ is $L^p$-minimal in $RE_1^+$ (cf.~(\ref{ClassicalMinimal})) for any $p$.
For $p=1$, $2$, we conjecture the $L^p$-minimality of $\e_n$ in $\K$ as an analogue of the case $n=1$.
\begin{conj}\label{L-minConj}
  For all $n$, $\e_n$ is $L^1$ and $L^2$-minimal in $\K$.
\end{conj}

We observe that our explicit unit $\e_n$ is $L^2$-minimal in $A_n$ for $1\leq n \leq 10$ by using Fincke--Pohst algorithm (qfminim command in PARI/GP).
Since $A_n=\K$ for $1\leq n \leq 6$, we obtain that $\e_n$ is $L^2$-minimal in $\K$.
For each $\e \in \K$, we see that $||l(\e)||_1=\log(\prod_{i=1}^{2^n}\max\{1, |\s^i(\e)|\})$, and
the value in $\log$ is called the Mahler measure of algebraic numbers.
Morisawa and Okazaki \cite{MO} investigate $\K$ by using the Mahler measure, and obtained a lower bound for $l(\K\setminus \{\pm1\})$ in $L^1$ norm as $2^{n-1}\log(2+\sqrt{5})$ (cf.~\cite[Lemma 3.2 and Theorem 5.3]{MO}).
They also obtained a lower bound in $L^2$ norm as $\sqrt{2^{n-1}}\log(2+\sqrt{5})$ (cf.~\cite[Lemma 2.5 (1)]{Morisawa-Okazaki}).
Note that these two lower bounds are processed into forms that fit our definitions.
We compare $||l(\e_n)||_p$ and lower bounds for small $n$ in \cref{Comparison}.
\begin{table}[h]
  \centering
  \begin{tabular}{|c||c|c||c|c|}
    $n$ & $||l(\e_n)||_1$ & $2^{n-1}\log(2+\sqrt{5})$ & $||l(\e_n)||_2$ & $\sqrt{2^{n-1}}\log(2+\sqrt{5})$ \\
    \hline
    $1$ & $1.76...$ & $1.44...$ & $1.76...$ & $1.44...$ \\
    $2$ & $3.22...$ & $2.88...$ & $2.35...$ & $2.04...$ \\
    $3$ & $6.28...$ & $5.77...$ & $3.54...$ & $2.88...$ \\
    $4$ & $12.47...$ & $11.54...$ & $5.04...$ & $4.08...$ \\
    $5$ & $24.89...$ & $23.09...$ & $7.20...$ & $5.77...$ \\
    $6$ & $49.76...$ & $46.19...$ & $10.22...$ & $8.16...$ \\
    $7$ & $99.52...$ & $92.39...$ & $14.48...$ & $11.54...$ \\
  \end{tabular}
  \caption{Comparison of $||l(\e_n)||_p$ and lower bounds}
  \label{Comparison}
\end{table}

In the following, by assuming that \cref{L-minConj} holds, we give upper bounds of $k_n=h_n/h_{n-1}$ for small $n$.
Let $m$ be a positive integer.
For a Lebesgue measurable set $S$ in $\R^m$, $\vol(S)$ denote the volume of $S$ in Lebesgue measure on $\R^m$.
For a complete lattice $L\subset \R^m$ with a basis ${\bf b}=\{b_1,...,b_m\}$, we define the volume of $L$ by the volume of the fundamental parallel body of ${\bf b}$,
namely, $\vol(L)=|\det([b_1...b_m])|$.
Then we have
\begin{equation}\label{volume}
  (\K:A_n)=\vol(l(A_n))/\vol(l(\K)).
\end{equation}
We use the following Blichfeldt's theorem.
Note that the following statement is processed into our settings.
\begin{theo}[{cf.~\cite[Theorem II, III]{Blichfeldt}}]\label{Blichfeldt}
  There exist $\e$, $\de \in\K\setminus\{\pm1\}$ such that
\[
  ||l(\e)||_2\leq \sqrt{\frac{2}{\pi}}\Gamma(2+2^{n-2})^{1/2^{n-1}}\vol(\K)^{1/2^{n-1}}
\]
and
\[
  ||l(\de)||_1\leq \sqrt{\frac{2^n}{\pi}}\Gamma(2+2^{n-2})^{1/2^{n-1}}\vol(\K)^{1/2^{n-1}},
\]
where $\Gamma$ is the gamma function.
\end{theo}
\cref{L-minConj} implies that $\e_n$ satisfies these inequalities.
Thus we have
\begin{equation}\label{VolumeBound1}
  \frac{\vol(A_n)}{\vol(\K)}\leq \frac{\vol(A_n)\sqrt{2^n/\pi}^{2^{n-1}}\Gamma(2+2^{n-2})}{||l(\e_n)||_1^{2^{n-1}}}
\end{equation}
and
\begin{equation}\label{VolumeBound2}
  \frac{\vol(A_n)}{\vol(\K)}\leq \frac{\vol(A_n)\sqrt{2/\pi}^{2^{n-1}}\Gamma(2+2^{n-2})}{||l(\e_n)||_2^{2^{n-1}}}.
\end{equation}
We compute the numerical values of the right-hand sides of (\ref{VolumeBound1}) and (\ref{VolumeBound2}) for each $n\leq 7$ in \cref{UpperBounds}.
\begin{table}[h]
  \centering
  \begin{tabular}{|c|r|r|}
    $n\backslash p$ & $1$ & $2$ \\
    \hline 
    $1$ &$1.06...$ &$1.06...$ \\
    $2$ &$1.35...$ &$1.27...$ \\
    $3$ &$2.51...$ &$1.55...$ \\
    $4$ &$14.44...$ &$4.89...$ \\
    $5$ &$4345.05...$ &$417.77...$ \\
    $6$ &$17992212754.52...$ &$147730099.26...$ \\
    $7$ &$14822653597271460343569281399.70...$ & $876387598588509574855259.98...$
  \end{tabular}
  \caption{Upper bounds for $k_n$ assuming \cref{L-minConj}}
  \label{UpperBounds}
\end{table}
Combining the table at $p=2$ and the fact that $h_n<10^9$ for all $n$ (cf.~\cite[Corollary 1.2]{Fukuda-Komatsu3}), we obtain $k_n=1$ for $1\leq n \leq 6$.

\begin{rem}
  By using Minkowski's convex body theorem for the $L^p$ norm open ball of the radius $||l(\e_n)||_p$, we also obtain upper bounds of $k_n$.
  In contrast to the discussion above, in this setting, the $L^1$-minimality of $\e_n$ gives more precise bound than the $L^2$-minimality.
\end{rem}

By these arguments, the resolution of \cref{L-minConj} contributes to Weber's conjecture and \cref{conj}.
However, determining the shortest vector in a lattice is generally a very difficult problem.
If we propose to approach \cref{L-minConj} by imitating the classical method in \cref{ClassicalMethod},
then we should establish ``{\it the best approximation to $X_n$ at $\Q(X_{n-1})$}''.

\section*{Acknowledgement}
The author would like to thank Tomokazu Kashio who gave the author much useful advice in this research and also Jun Ueki who gave the author much advice in writing this paper.
The author would like to thank his mother Yukari Yoshizaki for much support.
Finally, the author would like to express his special gratitude to the two dogs, Valon and Andy, who encouraged him in his hard times.
The author has been partially supported by JSPS KAKENHI Grant Number JP22J10004.

\bibliographystyle{alpha}
\bibliography{genPell}

\end{document}